\documentclass{amsart}

\makeatletter
 \def\LaTeX{\leavevmode L\raise.42ex
   \hbox{\kern-.3em\size{\sf@size}{0pt}\selectfont A}\kern-.15em\TeX}
\makeatother

\newcommand{\BibTeX}{{\rm B\kern-.05em{\sc
i\kern-.025emb}\kern-.08em\TeX}}
\newtheorem{col}{Corollary}[section]
\newtheorem{thm}{Theorem}[section]
\newtheorem{lem}[thm]{Lemma}

\theoremstyle{definition}
\newtheorem{defn}{Definition}

\numberwithin{equation}{section}

\begin{document}

\title[  Widths of Balls in Sobolev  spaces on  Manifolds]
{Kolmogorov and Linear  Widths of Balls in Sobolev  spaces on Compact Manifolds}

\maketitle
\begin{center}

\author {Daryl Geller}
\footnote{Department of Mathematics, Stony Brook University, Stony Brook, NY 11794-3651; (12/26/1950-01/27/2011)}

\author{Isaac Z. Pesenson }\footnote{ Department of Mathematics, Temple University,
 Philadelphia,
PA 19122; pesenson@temple.edu. The author was supported in
part by the National Geospatial-Intelligence Agency University
Research Initiative (NURI), grant HM1582-08-1-0019. }

\end{center}

\begin{abstract}
We determine upper asymptotic estimates of  Kolmogorov and linear  $n$-widths of unit balls in Sobolev  norms in $L_{p}$-spaces on smooth
compact Riemannian 
manifolds.  For  compact  homogeneous manifolds, we establish estimates which are asymptotically exact, for the natural
ranges of indices.
The proofs  heavily rely  on our previous results
such as: estimates for the near-diagonal localization of the kernels of  elliptic operators, Plancherel-Polya inequalities
on manifolds,  cubature formulas with positive coefficients and uniform
estimates on Clebsch-Gordon coefficients on general compact homogeneous manifolds.
\end{abstract}




\section{Introduction and the main results}

 %
Daryl Geller and I were working on this paper during the Summer and Fall of 2010. Sadly, Daryl Geller passed away suddenly in 
January of 2011. I will always remember him as a good friend and a wonderful 
mathematician. 

\bigskip

The goal of the paper is to determine asymptotic estimates of  Kolmogorov and linear $n$-widths of unit balls in
Sobolev  norms in $L_{p}({\bf M})$-spaces on a smooth compact (connected) Riemannian 
manifold ${\bf M}$.  For  compact homogeneous manifolds,
we establish estimates which are asymptotically exact, for the natural ranges of indices.  For compact homogeneous
manifolds, we also obtain some lower bounds for Gelfand widths, which will be discussed in section 5.

Let us recall \cite{LGM} that for a given subset $H$ of a normed linear space $Y$, the Kolmogorov $n$-width
$d_{n}(H,Y)$ is defined as
$$
d_{n}(H,Y)=\inf_{Z_{n}}\sup_{x\in H}\inf_{z\in Z_{n}}\|x-z\|_{Y}
$$
where $Z_{n}$ runs over all $n$-dimensional subspaces of $Y$. The linear $n$-width $\delta_{n}(H,Y)$  is defined as 
$$
\delta_{n}(H,Y)=\inf _{A_{n}}\sup_{x\in H}\|x-A_{n}x\|_{Y}
$$
where $A_{n}$ runs over all bounded operators $A_{n}: Y\rightarrow Y$ whose range has dimension $n$.
The Gelfand  $n$-width of a subset $H$ in a linear space $Y$ is defined by
$$
d^{n}(H,Y)=\inf_{Z^{n}}\sup\{\|x\|: \>\>x\in H\cap Z^{n}\},
$$
where the infimum is taken over all subspaces $Z^{n}\subset Y$ of codimension $\leq n$.
The width $d_{n}$ characterizes the best approximative possibilities by approximations by $n$-dimensional subspaces, the width $\delta_{n}$ characterizes the best approximative possibilities  of any $n$-dimensional linear method. The width $d^{n}$ plays a key role in questions about interpolation and reconstruction of functions.

In our paper the notation $S_{n}$ will stay for either Kolmogorov $n$-width $d_{n}$ or linear $n$-width $\delta_{n}$;
 the notation $s_{n}$ will be used for either $d_{n}$ or Gelfand $n$-width $d^{n};
\>\> S^{n}$ will be used for either $d_{n}, \>\> d^{n},$ or $\delta_{n}$.

One then has the following relations (see \cite{LGM}, pp. 400-403,):
\begin{equation}
\label{pupqdn}
S^n(H_{1}, Y) \leq S^n(H, Y),
\end{equation}
if $H_{1}\subset H$, and 
\begin{equation}
\label{pupqdn2}
d^{n}(H, Y)=d^{n}(H, Y_{1}), \>\>\>
S_n(H, Y) \leq S_n(H, Y_{1}), \>\>H\subset Y_{1}\subset Y, 
\end{equation}
where  $Y_{1}$ is  a subspace of $ Y$.
  Moreover, the following inequality holds
\begin{equation}
\label{linmax}
\delta_n(H,Y) \geq \max(d_n(H,Y),d^n(H,Y)).
\end{equation}

If $\gamma \in \bf R$, we write $S^n(H, Y) \ll n^{\gamma}$ to mean that one has the upper estimate $S^n(H, Y) \leq Cn^{\gamma}$ for $n > 0$.
(Here $C$ is independent  of $n$).  We say that the upper estimate is  exact
if also $S^n(H, Y) \geq cn^{\gamma}$ for $n > 0$, and in that case we write $S^n(H, Y) \asymp cn^{\gamma}$.

Let  $L_{q}=L_{q}({\bf M}),\> 1\leq q\leq\infty,$ be the regular Lebesgue  space constructed with the Riemannian density.
Let $L$ be an elliptic smooth second-order 
differential operator $L$ which is self-adjoint and positive definite in $L_{2}({\bf M})$, such as the Laplace-Beltrami operator $\Delta$.   For such an operator all the powers $L^{r}, \>\>r>0,$ 
are well defined on $C^{\infty}({\bf M})\subset L_{2}({\bf M})$ and continuously map $C^{\infty}({\bf M})$ into itself. 
Using duality every operator  $L^{r}, \>\>r>0,$ can be extended  to distributions on ${\bf M}$.
The Sobolev space $W_{p}^{r}=W_{p}^{r}({\bf M}),\> 1\leq p\leq\infty, \>\> r>0,$ is defined as the space of all 
$f\in L_{p}({\bf M}), 1\leq p\leq \infty$ for which the following graph norm is finite
\begin{equation}
\label{Sob}
\|f\|_{W^{r}_{p}({\bf M})}=\|f\|_{p}+\|L^{r/2} f\|_{p}.
\end{equation}
If $p \neq 1, \infty$, this graph norm is independent of $L$, up to equivalence, by elliptic regularity theory on compact manifolds.    If $p = 1$ or $\infty$ we will need to specify which operator $L$ we are using; some of our results will apply for $L$ general.  In fact,
for our results which apply to general ${\bf M}$, we can use any $L$.   For the results which apply only to homogeneous manifolds ${\bf M}$,  we will need to use a specific $L$, namely 
 the image  $\mathcal{L}$ (under the differential of the quasi-regular representation of $G$ in
 $L_{p}({\bf M}),\> 1\leq p\leq\infty$) 
  of a central element in the enveloping algebra of ${\bf g}$ which can be represented as a "sum of squares" (see section 3 below).   
Note, that if $G$ is compact and semi-simple then $\mathcal{L}$ will be the image of the Casimir operator in the enveloping algebra of the Lie algebra ${\bf g}$.
 For certain homogeneous manifolds the operator  $\mathcal{L}$ coincides  with the Laplace-Beltrami operator $\Delta$ of an invariant metric.
This happens, for example, when $\bf M$ is a symmetric compact homogeneous manifold of rank one
 (=two point compact homogeneous manifold) or when ${\bf M}$ is a compact Lie group $G$.

It is  important to remember that in all our considerations the inequality
 $
r>s\left(\frac{1}{p}-\frac{1}{q}\right)_+
$
with $s=dim \>{\bf M}$ will be satisfied. Thus, by the Sobolev embedding theorem the set  $B^{r}_{p}({\bf M})$ is a  subset of $L_{q}({\bf M})$. Moreover, since ${\bf M}$ is compact by the Rellich-Kondrashov theorem the embedding of $B^{r}_{p}({\bf M})$ into $L_{q}({\bf M})$  will be compact.

Our objective is to obtain asymptotic estimates of $S_{n}(H, L_{q}({\bf M}))$,  where  $H$   is the unit ball $B^r_p({\bf M})$ in the Sobolev space 
$W_{p}^{r}=W_{p}^{r}({\bf M}), 1\leq p\leq\infty,\>\>r>0,$   Thus,
$$
B^r_p=B^r_p({\bf M})=\left\{f\in W_{p}^{r}({\bf M})\>: \> \|f\|_{W_{p}^{r}({\bf M})}\leq 1\right\},
$$
We  also consider compact homogeneous manifolds ${\bf M}=G/K, \>\>G$ being a compact Lie group (with Lie algebra ${\bf g}$)
and $K$ its closed subgroup. In the case of compact homogeneous manifolds we are able to obtain
exact asymptotic estimates on $S_{n}(H, L_{q}({\bf M}))$. 
  
We set $s = \dim {\bf M}$. Let  as usual $p'=\frac{p}{p-1}$.  Our main results are the following three Theorems which are proved in sections 2, 4, and 5 respectively.
\begin{thm}
\label{basic}
(Basic upper estimate)
For any compact Riemannian manifold, any $L$, and for any $1\leq p,q\leq\infty, \> r>0$,  if  $S_{n}$ is either of $d_{n}$ or $\delta_{n}$ then the following holds 
\begin{equation}
\label{basicway}
S_n(B^r_p({\bf M}),L_q({\bf M})) \ll n^{-\frac{r}{s}+(\frac{1}{p}-\frac{1}{q})_+},
\end{equation}
provided that $-\frac{r}{s}+(\frac{1}{p}-\frac{1}{q})_+$, which we call the {\em basic exponent}, is negative.
\end{thm}

\begin{thm}(Improved estimates)
\label{impr}
Say ${\bf M}$ is a homogeneous manifold.  

\begin{enumerate}
\item
 Say $1\leq p \leq 2 \leq q\leq \infty$.  If $p = 1$, take $L = {\mathcal L}$.

Then one has the improved upper estimates
\begin{eqnarray}
d_n(B^r_p({\bf M}),L_q({\bf M})) & \ll & n^{-\frac{r}{s}+\frac{1}{p}-\frac{1}{2}}\ \ \ \ \ \ \ \ \ \ \ \ \mbox{if } r > s/p, \label{imp1}\\
\delta_n(B^r_p({\bf M}),L_q({\bf M})) & \ll & n^{-\frac{r}{s}+\frac{1}{p}-\frac{1}{2}}\ \ \ \ \ \ \ \ \ \ \ \ \mbox{if } q \leq p' \mbox{ and } r > s/p, \label{imp2}\\
\delta_n(B^r_p({\bf M}),L_q({\bf M})) & \ll & n^{-\frac{r}{s}+\frac{1}{2}-\frac{1}{q}}\ \ \ \ \ \ \ \ \ \ \ \ \mbox{if } q > p' \mbox{ and } r > s/q'. \label{imp3}
\end{eqnarray}
\item
 Say $2 \leq p \leq q\leq \infty$.  If $p = \infty$, take $L = {\mathcal L}$.

 Then one has the improved upper estimate
\begin{equation}
\label{imp4}
d_n(B^r_p({\bf M}),L_q({\bf M}))  \ll  n^{-\frac{r}{s}}\ \ \ \ \ \ \ \ \ \ \ \ \mbox{if } r > s/p. 
\end{equation}
\end{enumerate}
\end{thm}

\begin{thm}(Exact estimates)
\label{basicshp}
Assume  ${\bf M}$ is a homogeneous manifold.  
If $p = 1$ or $\infty$, take $L = {\mathcal L}$.
Then the four improved estimates listed in Theorem \ref{impr} are all exact.
In all other situations (i.e. $p \leq 2 \leq q$ is false, or $2 \leq p \leq q$ and $S_n = \delta_n$), if the basic
exponent is negative, then the basic upper estimate is exact.
\end{thm}
 
Thus  if ${\bf M}$ is a homogeneous
manifold we obtain exact asymptotic estimates for $d_{n}$ and $\delta_{n}$ for all $1\leq p,q\leq \infty$ and some restrictions on $r$. For general compact Riemannian manifolds we obtain only upper esimates.
Our results generalize some of the known estimates for
the particular case in which $\bf M$ is a compact symmetric space of rank one; these estimates were obtained in 
papers \cite{BKLT} and \cite{brdai}.
They, in turn generalized and extended results from \cite{BirSol}, \cite{Ho}, \cite{Kas}, \cite{Ma}, \cite{Ka1} and \cite{Ka2}.

Our main Theorems could be carried over to Besov spaces on manifolds using general results about interpolation of compact operators. 

The proofs of all the main results heavily exploit our estimates for the near-diagonal localization of the kernels of  
elliptic operators on compact manifolds  (see \cite{gpes} and section 2 below for the general case and  \cite{gm1}- \cite{gmmix} for the case of Laplace-Beltrami operator). These estimate allow one to decompose 
functions into bandlimited and fast decaying parts.

Of course, homogeneous compact manifolds are much ``better" than general compact Riemannian manifolds (section 3 below). But the main reason we obtain
exact asymptotic estimates is that in the case of homogeneous manifolds
we are able to find a uniform estimate on the 
number of non-zero Fourier coefficients of the product of two eigenfunctions of $\mathcal L$ 
(Theorem \ref{prodthm} bellow and Theorem 5.1 of \cite{gpes}). Note that this result is well known, say,
  for spherical harmonics and the corresponding non-zero coefficients are known as Wigner symbols.
 In a more general context  it is a problem of decomposing a tensor product of two representations of a compact 
Lie group into irreducible representations in which case the corresponding Fourier coefficients are known as the Clebsch-Gordon coefficients. 
  
  Out result about Clebsh-Gordon coefficients along with our positive cubature formula (Theorem \ref{cubformula} below,
 see also \cite{gpes}) allows us to discretize convolution integrals of eigenfunctions of $\mathcal {L}$ with zonal functions. 
It is the main technical trick in section 3 which produces improved estimates in the case of homogeneous manifolds.
  
  Note that the proof of  existence (even on general compact Riemannian manifolds) of cubature formulas   
with positive coefficients which are exact on eigensubspaces was prepared for in  \cite{Pes08} and published in \cite{gpes}. Plancherel-Polya-type inequalities on compact and non-compact manifolds were extensively developed in used in our previous papers.

In connection with cubature formulas and  Marcinkiewicz-Zygmund  (or Plancherel-Polya) inequalities on compact  manifolds we refer to  the following papers   \cite{FM1},  \cite{FM2},  \cite{HMS}, \cite{Mar}-- \cite{OCP2},  which contain a number of  important results.

\section{The Basic Upper Estimate on general compact Riemannian manifolds}

Let $({\bf M},g)$ be a smooth, connected, compact Riemannian manifold without boundary
with  Riemannian measure $\mu$.  We write $dx$ instead of $d\mu(x)$.
For $x,y \in {\bf M}$, let $d(x,y)$ denote the geodesic distance from $x$ to $y$.
We will frequently need the fact that if $M > s$, $x \in {\bf M}$ and $t > 0$, then
\begin{equation}
\label{intest}
\int_{\bf M} \frac{1}{\left[1 + (d(x,y)/t)\right]^M} dy \leq Ct^{s}
\end{equation}
with $C$ independent of $x$ or $t$.  (See, for example, the third bulleted point after Proposition 3.1 of 
\cite{gmcw}. (Note that in \cite{gmcw}, the dimension of the manifold in $n$, not $s$.)
 
Let $L$ be a smooth, positive, second order elliptic differential operator on ${\bf M}$,
whose principal symbol $\sigma_2(L)(x,\xi)$ is positive on  $\{(x,\xi) \in T^*{\bf M}:\ \xi \neq 0\}$.
In the proof of Theorems \ref{basic} and \ref{basicshp} we will take $L$ to be the Laplace-Beltrami
operator of the metric $g$,  but in the proof of Theorem \ref{impr} we will let $L$ be the Laplace operator
${\mathcal L}$, which we will discuss in the next section.
We will use the same notation
$L$ for the closure of $L$ from $C^{\infty}({\bf M})$ in $L_{2}({\bf M})$.
In the case $p=2$ this closure is a
self-adjoint positive definite operator on the space $L_{2}({\bf M})$.
The spectrum of this operator, say
$0=\lambda_{0}<\lambda_{1}\leq \lambda_{2}\leq ...$,
is discrete and approaches infinity.  Let
$u_{0}, u_{1}, u_{2}, ...$ be a corresponding
complete system of real-valued orthonormal eigenfunctions, and let
$\textbf{E}_{\omega}(L),\ \omega>0,$ be the span of all
eigenfunctions of $L$, whose corresponding eigenvalues
are not greater than $\omega$.    Since the
operator $L$ is of order two, the dimension
$\mathcal{N}_{\omega}$ of the space ${\mathbf E}_{\omega}(L)$ is
given asymptotically by Weyl's formula,
which says \cite{Hor68}, in sharp form:
For some $c > 0$,
\begin{equation}
\label{Weyl}
\mathcal{N}_{\omega}(L) = c\omega^{s/2} + O(\omega^{(s-1)/2}).
\vspace{.3cm}
\end{equation}
where $s=dim {\bf M}$.
Since $\mathcal{N}_{\lambda_l} = l+1$, we conclude that, for some constants $c_1, c_2 > 0$,
\begin{equation}
\label{lamest}
c_1 l^{2/s} \leq \lambda_l \leq c_2 l^{2/s}  
\end{equation}
for all $l$.
Since $L^m u_l = \lambda_l^m u_l$, and $L^m$ is an elliptic differential
operator of degree $2m$, Sobolev's lemma, combined with the last fact, implies that
for any integer $k \geq 0$, there exist $C_k, \nu_k > 0$ such that
\begin{equation}
\label{ulest}
\|u_l\|_{C^k({\bf M})} \leq C_k (l+1)^{\nu_k}. 
\end{equation}
 From these  facts one sees at once: 
\begin{equation}
\label{frec}
\begin{split}
\mbox{The mapping } \sum a_l u_l \rightarrow (a_l)_{l \geq 0} \mbox{ gives a Fr\'echet space isomorphism}\\
\mbox{ of }
C^{\infty}({\bf M}) \mbox{ with the space of rapidly decaying
sequences.}
\end{split}
\end{equation}
In particular, smooth functions are precisely those functions $F$ which can be written as $\sum_{l = 0}^{\infty} a_l u_l$, for certain
$a_l$ which decay rapidly.  If $r > 0$, $L^{r/2}f$ is defined to be the smooth function
$\sum_{l=0}^{\infty} a_l \lambda_l^{r/2} u_l$.  
In fact, from (\ref{frec}), we see that $L^{r/2}$ maps $C^{\infty}({\bf M})$ to itself continuously, and may thus
be extended by duality to a map on distributions.

Suppose  $F \in\mathcal{ S}(\bf{R}^{+})$, the space of restrictions to the 
nonnegative real axis of Schwartz functions on $\bf{R}$.  Using the spectral theorem,
one can define the bounded operator $F(t^{2}L)$ on $L_2({\bf M})$.  In fact, for $f \in L_2({\bf M})$,
\begin{equation}
\label{ft2F}
[F(t^{2}L)f](x) = \int K_t(x,y) f(y) dy,
\end{equation}
where 
\begin{equation}
\label{expout}
K_t(x,y) = \sum_l F(t^2\lambda_l)u_l(x)u_l(y) = K_t(y,x)
\end{equation}
as one sees easily by checking the case $F = u_m$.  
Using (\ref{expout}), (\ref{Weyl}), (\ref{lamest}) and (\ref{ulest}), one easily checks that $K_t(x,y)$ is smooth in $(x,y) \in 
{\bf M} \times {\bf M}$.  We call $K_t$ the kernel of $F(t^2L)$.  $F(t^2L)$ maps $C^{\infty}({\bf M})$
to itself continuously, and may thus be extended to be a map on distributions.  In particular
we may apply $F(t^2L)$ to any $F \in L_p({\bf M}) \subseteq L_1({\bf M})$ (where $1 \leq p \leq \infty$), and by Fubini's theorem
$F(t^2L)F$ is still given by (\ref{ft2F}).  

The following Theorem  about $K_t$ was proved in  \cite{gmcw} 
in the special case in which $L$ was $\Delta$.  In \cite{gpes} we argued that the result 
generalize to the situation in which $L$ is general. 
\begin{thm}
\label{nrdglc}
(Near-diagonal localization) Assume $fF\in \mathcal{ S}(\bf{R}^{+})$ (the space of restrictions to the 
nonnegative real axis of Schwartz functions on $\bf{R}$).
For $t > 0$, let $K_t(x,y)$ be the kernel of $F(t^{2}L)$.  Then: \\
(a)  If $F(0) = 0$.  Then for every pair of
$C^{\infty}$ differential operators $X$ $($in $x)$ and $Y$  $($in $y)$ on ${\bf M}$,
and for every integer $N \geq 0$, there exists $C_{N,X,Y}$ as follows.  Suppose
$\deg X = j$ and $\deg Y = k$.  Then
\begin{equation}
\label{diagest}
t^{s+j+k} \left|\left(\frac{d(x,y)}{t}\right)^N XYK_t(x,y)\right| \leq C_{N,X,Y} ,\>\>s=dim{\bf M},
\end{equation}
for all $t > 0$ and all $x,y \in {\bf M}$.\\
(b) For general $F$, the estimate (\ref{diagest}) at least holds for $0 < t \leq 1$.
\end{thm}

In this article, we will use the following corollaries of Theorem \ref{nrdglc}.

\begin{col}
\label{kersize}
Assume  $F \in \mathcal{ S}(\bf{R}^{+})$.
For $t > 0$, let $K_t(x,y)$ be the kernel of $F(t^{2}L)$.  Suppose that either:\\
(i) $F(0) = 0$, or\\
(ii) $F$ is general, but we only consider $0 < t \leq 1$.\\ \ \\
Then for some $C > 0$,
\begin{equation}
\label{kersizeway}
|K_t(x,y)| \leq \frac{Ct^{-s}}{\left[ 1+\frac{d(x,y)}{t} \right]^{s+1}}
\end{equation}
for all $t$ and all $x,y \in {\bf M}$.
\end{col}
{\bf Proof}  This is immediate from Theorem \ref{nrdglc}, with $X=Y=I$, if one
considers the two cases $N=0$ and $N=s+1$.

\begin{col}
\label{Lalphok}
Consider  $1 \leq \alpha \leq \infty$, with conjugate index $\alpha'$.
In the situation of Theorem \ref{kersize}, there is a constant $C > 0$ such that
\begin{equation}
\label{kint3a}
\left(\int |K_t(x,y)|^{\alpha}dy\right)^{1/\alpha} \leq Ct^{-s/\alpha'} \ \ \ \ \ \ \ \ \ \ \ \ \ \ \ \ \ \ \ \ \mbox{for all } x, 
\end{equation}
and
\begin{equation}
\label{kint4}
\left(\int |K_t(x,y)|^{\alpha}dx\right)^{1/\alpha} \leq Ct^{-s/\alpha'} \ \ \ \ \ \ \ \ \ \ \ \ \ \ \ \ \ \ \ \ \mbox{for all } y, 
\end{equation}
\end{col}
{\bf Proof}  We need only prove (\ref{kint3a}), since $K_t(y,x) = K_t(x,y)$.

If $\alpha < \infty$,
(\ref{kint3a}) follows from Corollary \ref{kersize}, which tells us that
\[
\int |K_t(x,y)|^{\alpha}dy 
\leq  C \int_{\bf M} \frac{t^{-s\alpha}}{\left[1 + (d(x,y)/t)\right]^{\alpha(s+1)}} dy \leq Ct^{s(1-\alpha)} \]
with $C$ independent of $x$ or $t$, by (\ref{intest}).

If $\alpha = \infty$, the left side of (\ref{kint3a}) is as usual to be interpreted as the $L^{\infty}$ norm
of $h_{t,x}(y) = K_t(x,y)$.  But in this case the conclusion is immediate from 
Corollary \ref{kersize}.  

This completes the proof.\\

We will use Corollary \ref{Lalphok} in conjunction with the following
fact.  We consider operators of the form $f \to {\mathcal K}f$
where
\begin{equation}
\label{kopdf}
({\mathcal K}f)(x) = \int K(x,y)f(y)dy, 
\end{equation}
where the integral is over ${\bf M}$, and where we are using Riemannian measure.
In all applications, $K$ will be continuous on ${\bf M} \times {\bf M}$,
and $F$ will be in $L_1({\bf M})$, so that ${\mathcal K}f$ will be a bounded continuous function.
The following generalization of Young's inequality holds:

\begin{lem}
\label{younggen}
Suppose $1 \leq p, \alpha \leq \infty$, and that $(1/q)+1 = (1/p)+(1/\alpha)$. 
Suppose that $c > 0$, and that
\begin{equation}
\label{kint1}
[\int |K(x,y)|^{\alpha}dy]^{1/\alpha} \leq c \ \ \ \ \ \ \ \ \ \ \ \ \ \ \ \ \ \ \ \ \mbox{for all } x, 
\end{equation}
and
\begin{equation}
\label{kint1}
[\int |K(x,y)|^{\alpha}dx]^{1/\alpha} \leq c \ \ \ \ \ \ \ \ \ \ \ \ \ \ \ \ \ \ \ \ \mbox{for all } y, 
\end{equation}
Then $\|{\mathcal K}f\|_q \leq c\|f\|_p$ for all $f \in L_p$.
\end{lem}

Now, let $\eta$ be a $C^{\infty}$ function on $[0,\infty)$ which equals $1$ on $[0,1]$,
and which is supported in $[0,4]$.  Define, for $x > 0$,
\[ \phi(x) = \eta(x/4) - \eta(x)\]
so that $\phi$ is supported in $[1,16]$.  For $j \geq 1$, we set 
\[ \phi_j(x) = \phi(x/4^{j-1}).\] 
We also set
$\phi_0 = \eta$, so that $\sum_{j=0}^{\infty} \phi_j \equiv 1$.  We claim:

\begin{lem}
\label{phijest}
(a) If $r > 0$, and $1 \leq p \leq q \leq \infty$, then there is a $C > 0$ such that
\begin{equation}
\label{phijestway}
\|\phi_j(L)f\|_q \leq C(2^{js})^{-\frac{r}{s}+\frac{1}{p}-\frac{1}{q}}\|f\|_{W_p^r}, 
\end{equation}
for all $f \in W_p^r({\bf M})$. In other words, the norm of $\phi_j(L)$, as an element
of ${\bf B}(W_p^r,L_q)$ (the bounded linear operators from $W_p^r$ to $L_q$), is
no more than $C(2^{js})^{-\frac{r}{s}+\frac{1}{p}-\frac{1}{q}}$.\\
(b) Suppose that $-\frac{r}{s}+\frac{1}{p}-\frac{1}{q} < 0$.  Then $\sum_{j=0}^{\infty} \phi_j(L)$
converges absolutely in ${\bf B}(W_p^r,L_q)$, to the identity operator on $W_p^r$.
\end{lem}
{\bf Proof} (a) 
Define, for $x > 0$,
\[ \psi(x) = \phi(x)/x^{r/2} \]
so that $\psi$ is supported in $[1,16]$.  For $j \geq 1$, we set 
\[ \psi_j(x)=\psi(x/4^{j-1}),\] 
so that $\phi_j(x) = 2^{-(j-1)r}\psi_j(x)x^{r/2}$. 

Accordingly, if $f$ is a distribution on ${\bf M}$, for $j \geq 1$,
\[ \phi_j(L)f = 2^{-(j-1)r}\psi_j(L)(L^{r/2}f), \]
in the sense of distributions.   
If now $f \in  W_p^r$, so that $L^{r/2}f \in L_p$, we see from Lemma \ref{Lalphok} with $t = 2^{-j}$, and 
from Lemma \ref{younggen}, that if $(1/q)+1 = (1/p)+(1/\alpha)$, then
\[ \|\phi_j(L)f\|_q \leq C2^{-jr}2^{js/\alpha'}\|L^{r/2}f\|_p \leq C (2^{js})^{-\frac{r}{s}+\frac{1}{p}-\frac{1}{q}}\|f\|_{W_p^r}, \]
as desired.

For (b), we note that by (a), $\sum_{j=0}^{\infty} \phi_j(L)$
converges absolutely in ${\bf B}(W_p^r,L_q)$.  It converges to the identity on smooth functions, hence in the sense
of distributions.  Hence we must have $\sum_{j=0}^{\infty} \phi_j(L) =  I$ in ${\bf B}(W_p^r,L_q)$.  This completes the proof.\\

{\bf Proof of Theorem \ref{basic}}  Since in general $d_n \leq \delta_n$, it suffices to prove the basic upper estimate for $\delta_n$.  If $q \leq p$,
then surely $\delta_n(B^r_p,L_q) \leq C\delta_n(B^r_p,L_p)$.  Since the basic upper estimate is the same for all $q$ with $q \leq p$, we may
as well assume then that $q = p$.  In short, we may assume $q \geq p$.

Let $\eta_M(x) = \eta(x/4^{M-1})$; then $\sum_{j=0}^{M-1} \phi_j = \eta_M$, which is supported in $[0,4^M]$.  Examining the kernel of 
$\eta_M(L)$ (see (\ref{expout})), we see that $\eta_M(L): W_p^r \to E_{4^M}(L)$.  By Weyl's theorem (\ref{Weyl}), 
there is a positive integer $c$ such that the 
dimension of $E_{4^M}(L)$ is at most $c2^{Ms}$ for every $M$.  We see then by Lemma \ref{phijest} that

\[ \delta_{c2^{Ms}}(B_p^r(L),L^q) \leq \|I-\eta_M(L)\| \leq \sum_{j=M}^{\infty} \|\phi_j(L)\| \leq 
$$
$$
\sum_{j=M}^{\infty} C(2^{js})^{-\frac{r}{s}+\frac{1}{p}-\frac{1}{q}} \leq C(2^{Ms})^{-\frac{r}{s}+\frac{1}{p}-\frac{1}{q}}
\leq  C(c2^{Ms})^{-\frac{r}{s}+\frac{1}{p}-\frac{1}{q}}
, \]
where all norms are taken in ${\bf B}(W_p^r,L_q)$.  This proves the basic upper estimate for\\ $n \in A :=
\{ c2^{Ms}: M \geq 1\}$.  For any $n \geq c2^s$ we may find $m \in A$ with $m \leq n \leq 2^s m$, and surely
$\delta_n \leq \delta_m$.  This gives the basic upper estimate for all $n$, and completes the proof.\\

We close this section with a result related to Theorem \ref{nrdglc}.  This result is an essential ingredient of its proof
(see \cite{gmcw} and section 7 of \cite{gpes}).
We will utilize it in the proof of Theorem \ref{basicshp} in section 5.  

\begin{thm}
\label{wavprop}
Suppose   $h(\xi) = F(\xi^2) \in \mathcal{S}(\bf{R})$ is even, and satisfies $\ supp\  \hat{h} \subseteq(-1,1)$.  
For $t > 0$, let
$K_t(x,y)$ be the kernel of $h(t\sqrt{L}) = F(t^2L)$.  Then for some $C_0 > 0$, if $d(x,y) > C_0t$,
then $K_t(x,y) = 0$.
\end{thm}

 \section{Harmonic Analysis on Compact homogeneous manifolds}

In this section we review and extend our previous results 
about Plancherel-Polya inequalities and cubature formulas on manifolds. We also reprove  our result which gives
 an estimate of the dimension of the eigenspace of the Casimir operator that contains the product of two of its eigenfunctions.

It is important to note that all the statements below from Lemma \ref{covL} to Theorem \ref{cubformula} hold true
 for all compact Riemannian manifolds and self-adjoint elliptic second order differential operators on them. Only in
 Theorems \ref{genp} and \ref{prodthm} we use the fact that $\bf {M}$ is a homogeneous manifold and $\mathcal{L}$ is the Casimir operator.

A \textit{homogeneous} compact manifold ${\bf M}$ is a
$C^{\infty}$-compact manifold on which a compact
Lie group $G$ acts transitively. In this case ${\bf M}$ is necessary of the form $G/K$,
where $K$ is a closed subgroup of $G$. The notation $L_{p}({\bf M}),
1\leq p\leq \infty,$ is used for the usual Banach  spaces
$L_{p}({\bf M},dx), 1\leq p\leq \infty$, where $dx$ is an invariant
measure.

Every element $X$ of the (real) Lie algebra of $G$ generates a vector
field on ${\bf M}$, which we will denote by the same letter $X$.  The translations along integral curves of such
vector fields $X$ on ${\bf M}$  can be identified with a one-parameter
group of diffeomorphisms of ${\bf M}$, which is usually denoted as $\exp
tX, -\infty<t<\infty$. At the same time, the one-parameter group
$\exp tX, -\infty<t<\infty,$ can be treated as a strongly
continuous one-parameter group of operators acting on the space $L_{p}({\bf M}),
1\leq p\leq \infty$. 
 The
generator of this one-parameter group will be denoted by $D_{X,p}$. 
According to the general theory of one-parameter groups in Banach
spaces   the operator $D_{X,p}$ is a closed
operator on every $L_{p}({\bf M}), 1\leq p\leq \infty$.   In order to
simplify notation, we will often write $D_{X}$ in place of
$D_{X, p}$.

If $\textbf{g}$ is the Lie algebra of a compact Lie group $G$ then it is a direct sum
$\textbf{g}=\textbf{a}+[\textbf{g},\textbf{g}]$, where
$\textbf{a}$ is the center of $\textbf{g}$, and
$[\textbf{g},\textbf{g}]$ is a semi-simple algebra. Let $Q$ be a
positive-definite quadratic form on $\textbf{g}$ which, on
$[\textbf{g},\textbf{g}]$, is opposite to the Killing form. Let
$X_{1},...,X_{d}$ be a basis of
$\textbf{g}$, which is orthonormal with respect to $Q$.
 Since the form $Q$ is $Ad(G)$-invariant, the operator
$$
-X_{1}^{2}-X_{2}^{2}-\    ... -X_{d}^{2},    \ d=dim\ G
$$
is a bi-invariant operator on $G$. This implies in particular that
the
   corresponding operator on $L_{p}({\bf M}), \      1\leq p\leq\infty,$
\begin{equation}
\mathcal{L}=-D_{1}^{2}- D_{2}^{2}- ...- D_{d}^{2}, \>\>\>
       D_{j}=D_{X_{j}}, \        d=dim \ G,\label{Laplacian}
\end{equation}
\textit{commutes} with all operators $D_{j}=D_{X_{j}}$. We will use this elliptic operator
$\mathcal{L}$ as our $L$ in the rest of the paper. 
However, as we discussed in the introduction, in all of the results of this section except 
for Theorem \ref{prodthm} below, one could use other $L$.

In the rest of the paper, the notation $\mathbb{D}=\{D_{1},...,
D_{d}\},\>\>\> d=dim \ G,$ will be used for the differential operators
on $L_{p}({\bf M}), 1\leq p\leq \infty,$ which are involved in the
formula (\ref{Laplacian}).

When discussing Sobolev spaces on ${\bf M}$, it is often crucial to utilize a positive 
elliptic operator, and in this paper, as in \cite{gpes},
we will use the {\em Laplace operator} $\mathcal{L}$.  Our results, which 
require only the definitions of Sobolev spaces and of $L_p$ to state, do not make
explicit mention of $\mathcal{L}$.

As we remarked in \cite{gpes}, there are 
situations in which the operator $\mathcal{L}$ is, or is proportional to, the
Laplace-Beltrami operator  of an invariant metric on ${\bf M}$. This
happens for example, if ${\bf M}$ is a $d$-dimensional torus, a compact semi-simple
Lie group, or a compact symmetric space of rank one.

Let $B(x,r)$ be a metric ball on ${\bf M}$ whose center is $x$ and
radius is $r$. 
The following lemma holds for any compact manifolds and can be found  in  \cite{Pes00},
\cite{Pes04a}.

\begin{lem}
\label{balls0}
There exists
a natural number $N_{{\bf M}}$, such that  for any sufficiently small $\rho>0$,
there exists a set of points $\{y_{\nu}\}$ such that:
\begin{enumerate}
\item the balls $B(y_{\nu}, \rho/4)$ are disjoint,

\item  the balls $B(y_{\nu}, \rho/2)$ form a cover of ${\bf M}$,

\item  the multiplicity of the cover by balls $B(y_{\nu}, \rho)$
is not greater than $N_{{\bf M}}.$
\end{enumerate}\label{covL}
\end{lem}

The following notion is involved in formulations of several our results.

\begin{defn}
Any set of points ${\bf M}_{\rho}=\{y_{\nu}\}$ which is as described in
Lemma \ref{covL} will be called a metric
$\rho$-lattice.\label{D1}
\end{defn}

The next two theorems were proved in \cite{Pes00}- \cite{Pes04b},
for a Laplace-Beltrami operator in $L_{2}(\bf {M})$ on a Riemannian manifold $\bf {M}$ of
bounded geometry,  but their proofs go through for any
elliptic second-order differential operator in  $L_{p}({\bf M}), 1\leq p\leq \infty$. 
In what follows the notation $s=dim\  {\bf M}$ is used.

\begin{thm}  For any $1\leq p\leq \infty$ there exist constants $\>\>C_{1}=C_{1}({\bf M}, p)>0\>\>$ and
 $\>\>\>
\rho_{0}({\bf M}, p)>0,$ such that for any  natural number $l>s/p$,
any $0<\rho<\rho_{0}({\bf M})$,  and any $\rho$-lattice
${\bf M}_{\rho}=\{x_{k}\}$,  the following inequality holds:
$$\left(\sum _{x_{k}\in {\bf M}_{\rho}}|f(x_{k})|^{p}\right)^{1/p}\leq
 C_{1}\rho^{-s/p}\left(\|f\|_{p}+\|\mathcal L ^{l/2}f\|_{p}\right),
 $$\label{MT}
for all  $f\in W_{p}^{l}({\bf M}), \>\>\> l>s/p,\>\>\> l\in
\mathbb{N}.$\label{T1}
\end{thm}

\begin{thm} For any $1\leq p\leq \infty$ there exist constants
$C_{2}=C_{2}({\bf M}, p)>0,$ and $\>\>\>
\rho_{0}({\bf M}, p)>0,$ such that for any  natural $m>s/p$,
any $0<\rho<\rho_{0}({\bf M})$, and any $\rho$-lattice
${\bf M}_{\rho}=\{x_{k}\}$ the following  inequality holds

\begin{equation}
\|f\|_{W_{p}^{m}({\bf M})}\leq C_{2}\left\{\rho^{s/p}\left(\sum_{x_{k}\in M_{\rho}}
|f(x_{k})|^{p}\right)^{1/p}+\rho^{2m}\|\mathcal{L}^{m}f\|_{L_{p}(\bf{M})}\right\},\label{rightPP1}
\end{equation}
where $m\in \mathbb{N},\>\>\> m>s/p.$

\end{thm}

Using  the constant $C_{2}({\bf M},p)$ from this Theorem, we
define another constant
\begin{equation}
c_{0}=c_{0}({\bf M}, p)=\left(2C_{2}({\bf M},  p)\right)^{-1/2m_{0}},\label{const}
\end{equation}
where $m_{0}=1+[s/p],\>\>\>s=dim {\bf M}$.
Since  $\mathcal{L}$ is an elliptic second order differential operator which is  positive definite and self-adjoint in the corresponding space $L_{2}(\bf{M})$  it has a discrete spectrum $0=\lambda_{0}<\lambda_{1}\leq \lambda_{2}\leq ...$. Let ${\mathbf E}_{\omega}(\mathcal{L})$ be the span of the corresponding eigenfunctions whose eigenvalues $\leq \omega$.
 As one can easily verify the norm of $\mathcal{L}$ on the subspace  ${\mathbf E}_{\omega}(\mathcal L)$ is exactly $\omega$.  In particular one has the following Bernstein-type inequality \cite{Pes08}
 \begin{equation}
 \label{Bern-2}
 \|\mathcal{L}^{\alpha}f\|\leq \omega^{\alpha}\|f\|,\>\>\alpha\in \mathbf{ R},
 \end{equation}
 for all $f\in {\mathbf E}_{\omega}(\mathcal{L})$.
 This fact and the previous two theorems imply the
following Plancherel-Polya-type inequalities.  Such
inequalities are also known as Marcinkewicz-Zygmund inequalities.

\begin{thm}
\label{PP2}
There exist constants $\>\>\>c_{1}=c_{1}({\bf M})>0,\>\>$
$c_{2}=c_{2}({\bf M})>0,$
 and
$\>\>c_{0}=c_{0}({\bf M})>0,$ such that for any $\omega>0$, and for
every metric $\rho$-lattice ${\bf M}_{\rho}=\{x_{k}\}$ with $\rho=
c_{0}\omega^{-1/2}$, the following Plancherel-Polya inequalities
hold:
\begin{equation}
c_{1}\left(\sum_{k}|f(x_{k})|^{2}\right)^{1/2}\leq\rho^{-s/2}\|f\|_{L_{2}({\bf M})}
\leq c_{2}\left(\sum_{k} |f(x_{k})|^{2}\right)^{1/2}, \label{completePlPo100}
\end{equation}
for all $f\in {\mathbf E}_{\omega}(\mathcal{L})$ and $s=\dim \  {\bf M}$. 
\label{completePlPo2}

\end{thm}

In \cite{gpes} we proved existence of cubature formulas,
which are exact on $ {\mathbf E}_{\omega}({\bf M})$,
and which have positive coefficients of the "right" size:

\begin{thm} 
\label{cubformula}
There exists  a  positive constant $a_{0}$,    such  that if  $\rho=a_{0}(\omega+1)^{-1/2}$, then
for any $\rho$-lattice $M_{\rho}$, there exist strictly positive coefficients $\lambda_{x_{k}}>0, 
 \  x_{k}\in {\bf M}_{\rho}$, \  for which the following equality holds for all functions in $ {\mathbf E}_{\omega}({\bf M})$:
\begin{equation}
\label{cubway}
\int_{{\bf M}}fdx=\sum_{x_{k}\in M_{\rho}}\lambda_{x_{k}}f(x_{k}).
\end{equation}
Moreover, there exists constants  $\  c_{1}, \  c_{2}, $  such that  the following inequalities hold:
\begin{equation}
c_{1}\rho^{s}\leq \lambda_{x_{k}}\leq c_{2}\rho^{s}, \ s=dim\   {\bf M}.
\end{equation}
\end{thm}
Our nearest goal is to prove the following key result which extends the Plancherel-Polya inequalities to general $1\leq p\leq \infty$.  The proof of this result uses in a crucial way the fact that
${\mathcal L}$ is the Laplace operator on the homogeneous manifold ${\bf M}$.   The lemma below can be found in \cite{Pes08} . We include the idea of its proof in order to correct a mistake made in \cite{Pes08}.

\begin{lem}
For any $1\leq p\leq \infty$ there exists a constant $C(\bf M)$ such that for $f\in {\mathbf E}_{\omega}(\mathcal{L})$ the following Bernstein inequality holds
\begin{equation}
\label{Bern}
\|\mathcal{L}^{m}f\|_{p}\leq (d\omega)^{m} \|f\|_{ p},\>\>m=0,1,2,....,
\end{equation}
where ${\bf M}=G/K$ and $d$ is  dimension of the group $G$.

\end{lem}

\begin{proof}It was shown in \cite{Pes08} the following equality takes place,  $k\in \mathbb{N}$,
\begin{equation}
\label{Main1}
\|\mathcal{L}^{k/2}f\|_{2}^{2}=\sum_{1\leq i_{1},...,i_{k}\leq
d}\|D_{i_{1}}...D_{i_{k}}f\|_{2}^{2}.
\end{equation}
Thus, according to (\ref{Main1})  we have the following inequality for all $f \in  \mathbf{E}
_{\omega}(\mathcal{L}), \>\>m\in \mathbb{N}$
\begin{equation}
\label{mixedderiv}
\|D_{i_{1}}...D_{i_{m}}f\|_{2}\leq 
\left(\sum_{1\leq
i_{1},...,i_{m}\leq d}\|D_{i_{1}}...D_{i_{m}}f\|_{2}^{2}\right)^{1/2}=
 \|\mathcal{L}^{m/2}f\|_{2},
\end{equation}
and using (\ref{Bern-2})  we obtain
\begin{equation}
 \|D_{i_{1}}...D_{i_{m}}f\|_{2}\leq 
\omega^{m/2}\|f\|_{2},\>\>f \in W^{m}_{2}({\bf M}),\>\>\>m\in \mathbb{N}.
\end{equation}
In particular, for every $1\leq j\leq d,\>\>\>d=dim \>G,$ we have
\begin{equation}
\|D_{j}^{m}f\|_{2}\leq \omega^{m/2}\|f\|_{2},\>\>\>f \in  \mathbf{E}
_{\omega}(\mathcal{L})\>\>\>m\in \mathbb{N}.
\end{equation}
Because $ \mathbf{E}_{\omega}(\mathcal{L})$ is a finite dimensional subspace any two norms on this subspace are equivalent.  Thus for any $1\leq p\leq \infty$ there exists a constant $C_{p}$ such that the following inequality holds
\begin{equation}
\label{PreBp}
\|D_{j}^{m}f\|_{p}\leq C_{p}\omega^{m/2}\|f\|_{p}
\end{equation}
for all $f \in  \mathbf{E}_{\omega}(\mathcal{L})$. 
Note, that since $\mathcal{L}$ commutes with every $D_{j}$ the space $ \mathbf{E}_{\omega}(\mathcal{L})$ is invariant under every operator $D_{j},\>\>\>1\leq j\leq d$.
Since $D_{j}$ generates a one-parameter group of isometries in $L_{p}({\bf M}), \>\>1\leq p\leq \infty,$ we  can  use the Lemma 5.2 in \cite{Pes00} which implies 
\begin{equation}
\|D_{j}^{m}f\|_{p}\leq \omega^{m/2}\|f\|_{p},\>\>f \in  \mathbf{E}_{\omega}(\mathcal{L}).
\end{equation}
Finally we obtain the following inequality

$$
\|\mathcal{L}^{m}f\|_{p}=\|(D_{1}^{2}+...+D_{d}^{2})^{m}f\|_{p}\leq (d\omega)^{m}\|f\|_{p},\>\>\>m\in \mathbb{N},\>\>f \in  \mathbf{E}_{\omega}(\mathcal{L}).
$$
 The lemma is proved.
 \end{proof}
\begin{thm}
\label{genp}
Say $1 \leq p \leq \infty$.  Then
there exist constants $\>\>\>c_{1}=c_{1}({\bf M}, p)>0,\>\>$
$c_{2}=c_{2}({\bf M}, p)>0,$
 and
$\>\>c_{0}=c_{0}({\bf M}, p)>0,$ such that for any $\omega>0$, and for
every metric $\rho$-lattice ${\bf M}_{\rho}=\{x_{k}\}$ with $\rho=
c_{0}\omega^{-1/2}$, the following Plancherel-Polya inequalities
hold:
\begin{equation}
c_{1}\left(\sum_{k}|f(x_{k})|^{p}\right)^{1/p}\leq\rho^{-s/p}\|f\|_{L_{p}({\bf M})}
\leq c_{2}\left(\sum_{k} |f(x_{k})|^{p}\right)^{1/p}, \label{completePlPo100}
\end{equation}
for all $f\in {\mathbf E}_{\omega}(\mathcal{L})$ and $s=\dim \  {\bf M}$.  (Here
one uses the usual interpretations of the inequalities when $p = \infty$.)
\label{completePlPo200}
\end{thm}
\begin{proof}
Since $\mathcal{L}$ is an  elliptic second-order differential operator on a compact manifold which is self-adjoint and positive definite in $L_{2}(\bf{M})$
 the norm on the Sobolev space $W_{p}^{2m}(\bf{M})$ is equivalent to the norm $\|f\|_{p}+\|\mathcal{L}^{m}f\|_{p}$.  Theorem \ref{MT}  with $l=2m$ 
implies
$$
\left(\sum _{x_{k}\in {\bf M}_{\rho}}|f(x_{k})|^{p}\right)^{1/p}\leq
 C_{1}\rho^{-s/p}\left( \|f\|_{p}+\|\mathcal{L}^{m}f\|_{p} \right) ,
 $$
for all  $f\in W_{p}^{2m}({\bf M}), \>\>\>2m> \frac{s}{p}.$
The  Bernstein inequality (\ref{Bern}) shows that if $m_{0}=\left[\frac{s}{p}\right]+1$ , then there exists a constant $a({\bf M})$ such that for all $f\in {\mathbf E}_{\omega}(\mathcal{L})$ 
  $$
  \|f\|_{p}+\|\mathcal{L}^{m_{0}/2}f\|_{p}\leq \left(a({\bf M})(1+\omega)^{m_{0}/2}\right)\|f\|_{p},\>\>\omega\geq 0.
 $$
 Thus we proved the inequality 
 \begin{equation}
 \left(\sum _{x_{k}\in {\bf M}_{\rho}}|f(x_{k})|^{p}\right)^{1/p}\leq
 C_{1}\rho^{-s/p}\|f\|_{p}, \>\>\\\ f\in {\mathbf E}_{\omega}(\mathcal{L}),
 \end{equation}
 where $C_{1}=a(\bf M)(1+\omega)^{m_{0}/2}$.
  To prove the opposite inequality we use (\ref{rightPP1})  and  (\ref{Bern}) to  obtain
\begin{equation}
\label{x}
\|f\|_{p}\leq C_{2}\rho^{s/p}\left(\sum_{x_{k}\in M_{\rho}}
|f(x_{k})|^{p}\right)^{1/p}+C_{2}d^{m_{0}}\rho^{2m_{0}}\omega^{m_{0}}\|f\|_{p},
\end{equation}
where  $f\in {\mathbf E}_{\omega}(\mathcal{L})$ and $    m_{0}=\left[\frac{s}{p}\right]+1           $. Now we fix the following value for $\rho$
$$
\rho=\left(\frac{1}{2}C_{2}^{-1}d^{-m_{0}}\right)^{1/2m_{0}}\omega^{-1/2}=c_{0}\omega ^{-1/2}.
$$
With such $\rho$ the factor in the front of the last term in (\ref{x}) is exactly $1/2$. Thus, this term can be moved to the left side of the formula (\ref{x}) to obtain
\begin{equation}
\label{xx}
2\|f\|_{p}\leq C_{2}\rho^{s/p}\left(\sum_{x_{k}\in{\bf  M}_{\rho}}
|f(x_{k})|^{p}\right)^{1/p}.
\end{equation}
In other words, we obtain the inequality
$$
\|f\|_{p}\leq c_{2}\rho^{s/p}\left(\sum_{x_{k}\in {\bf M}_{\rho}}
|f(x_{k})|^{p}\right)^{1/p},
$$
where $c_{2}=C_{2}/2$.
The theorem is proved.
\end{proof}

Our reason for using Casimir $\mathcal{L}$
instead of the Laplace-Beltrami operator or another elliptic operator on 
${\bf M}$, is that we can prove the following important  fact:

\begin{thm}
\label{prodthm}(Theorem 5.1 of \cite{gpes}:)
If ${\bf M}=G/K$ is a compact homogeneous manifold and $\mathcal{L}$
is defined as in (\ref{Laplacian}), then for any $f$ and $g$ belonging
to ${\mathbf E}_{\omega}(\mathcal{L})$,  their product $fg$ belongs to
${\mathbf E}_{4d\omega}(\mathcal{L})$, where $d$ is the dimension of the
group $G$.
\end{thm} 

We also need some basic facts about eigenfunctions on homogeneous manifolds.
\begin{lem}
\label{eigensum}
Let ${\bf M} = G/K$ be a compact homogeneous manifold. 
Assume $\lambda > 0$ is an eigenvalue of ${\mathcal L}$, and let $\{v_1,\ldots,v_m\}$
be an orthonormal basis of real-valued functions for $V_{\lambda}$,
the eigenspace of ${\mathcal L}$ for the eigenvalue $\lambda$.  
Say $x \in {\bf M}$.  Then
\begin{equation}
\label{eigensumway}
\sum_{k=1}^m [v_k(x)]^2 = \frac{\dim V_{\lambda}}{\mu({\bf M})}.
\end{equation}
\end{lem}
\begin{proof} 
For fixed $x \in {\bf M}$, let $Z_x(y)= \sum_{k=1}^m v_k(x)v_k(y)$ for 
$y \in {\bf M}$.  $Z_x$ is the
unique element of $V_{\lambda}$ satisfying 
$$
F(x) = \int F(y)Z_x(y)dy
$$
 for each $F \in V_{\lambda}$.
Since $dy$ and $V_{\lambda}$ are invariant under $G$, one sees from this that 
$Z_x(y) = Z_{g\cdot x}(g\cdot y)$ for each $g \in G$.  Since $G$ acts transitively
on ${\bf M}$, we see in particular from this that 
$Z_x(x) = \sum_{k=1}^m [v_k(x)]^2$ is independent of $x$.  Accordingly
\[ {\mu({\bf M})}\sum_{k=1}^m [v_k(x)]^2 = \int_{\bf M} \sum_{k=1}^m [v_k(u)]^2 du
= {\dim V_{\lambda}} \]
as desired.

\end{proof}

Recall that we have denoted the spectrum of $L$ by
$0=\lambda_{0}<\lambda_{1}\leq \lambda_{2}\leq ...$,
and we have let
$u_{0}, u_{1}, u_{2}, ...$ denote a corresponding
complete system of real-valued orthonormal eigenfunctions.

\begin{col}
\label{ulsq}
Suppose $0 < a < b$.  Then for any $x \in {\bf M}$,
\begin{equation}
\label{ulsqway}
\sum_{a/t^2 < \lambda_l \leq b/t^2} |u_l(x)|^2 \asymp t^{-s},
\end{equation}
as $t \to 0^+$, with constants independent of $x$ or $t$.
\end{col}
\begin{proof}
By Proposition \ref{eigensum} and (\ref{Weyl}), we have
\[ \sum_{a/t^2 < \lambda_l \leq b/t^2} |u_l(x)|^2 = \mu({\bf M})^{-1}[\mathcal{N}_{b/t^2}(L) - \mathcal{N}_{a/t^2}(L)] \asymp t^{-s} \]
as claimed.
\end{proof}

This then allows us to prove the following improvement on Corollary \ref{Lalphok}, for homogeneous manifolds.  
\begin{thm}
\label{Lalphimp}
In the situation of Corollary \ref{kersize} and Corollary \ref{Lalphok}, say that $f \neq 0$, and
${\bf M}$ is a homogeneous manifold, and $1 \leq \alpha \leq \infty$.  Then we actually have that
\begin{equation}
\label{kint3}
[\int |K_t(x,y)|^{\alpha}dy]^{1/\alpha} \asymp t^{-s/\alpha'}
\end{equation}
with constants independent of $x$ or $t$, as $t \to 0$.
\end{thm}
\begin{proof}
 By Corollary \ref{Lalphok}, we need only prove the lower bounds.  
First we handle the case $\alpha = 2$.  Since $f$ is not identically zero, we may find $0 < a < b$ and 
$c > 0$, such that $|f| \geq c$ on $[a,b]$.   By (\ref{expout}) and Corollary \ref{ulsq}, we have that
\[ \int |K_t(x,y)|^2 dy = \sum_l |f(t^2\lambda_l)|^2 |u_l(x)|^2 \geq 
\sum_{l: a/t^2 \leq \lambda_l \leq b/t^2} |f(t^2\lambda_l)|^2 |u_l(x)|^2 
\geq
$$
$$
 c^2\sum_{a/t^2 < \lambda_l \leq b/t^2} |u_l(x)|^2  \gg t^{-s} \]
as $t \to 0^+$, with constants independent of $x$ or $t$.  This establishes the case $\alpha = 2$.

The lower bounds for $\alpha = 1,\infty$ now follow at once from the simple
general inequality
\begin{equation}  \label{21inf}
\|f\|_2^2 \leq \|f\|_1 \|f\|_{\infty},
\end{equation}
the lower bound for $\alpha = 2$, and the upper bounds for $\alpha = 1,\infty$, as applied to $f(y) = K_t(x,y)$.

For the lower bounds for other $\alpha$,
we note that if $q < 2 < r$, and if $0 < \theta < 1$
is the number with $1/2 = \theta/q + (1-\theta)/r$, then one has the general
inequality
\begin{equation}  \label{q2r}
\|f\|_2 \leq \|f\|_q^{\theta} \|f\|_{r}^{1-\theta}.
\end{equation}
If $\alpha > 2$, the lower bound follows, after a brief computation, from (\ref%
{q2r}) in the case $q = 1$, $r = \alpha$, and the lower bounds for $2$ and $1$.
If $\alpha < 2$, the lower bound follows, after a briefer computation, from (\ref%
{q2r}) in the case $q = \alpha$, $r = \infty$, and the lower bounds for $2$ and $%
\infty$. This completes the proof.

\end{proof}

\section{The improved upper estimate}

For $x = (x_{1}, . . . , x_{m}) \in \mathbb{R}^{m}$, we define as usual $\|x\|^{m}_{p} =
(\sum_{i=1}^m |x_{i}|^{p})^{1/p}$ for $1 \leq p < \infty$, and $\|x\|^{m}_{\infty}
= \max_{1 \leq i \leq m} |x_{i}|$.  
 We denote by $\ell^m_p$
the set of vectors $x \in \mathbb{R}^{m}$ endowed
with the norm $\|\cdot\|_{\ell^m_p}$ and $b^m_p$
the unit ball of $\ell^m_p$. Given $1 \leq p \leq \infty$ and an integer
$N \geq 0$, we denote by $\mathbb{B}^{p}_{N} = \mathbb{B}^{p}_{N}({\bf M})$ 
the class of all functions $f \in {\bf E}_{N^{2}}(\mathcal L)$ such
that 	$\|f\|_p \leq 1$.

The proof of Theorem \ref{impr} will rely on the following:

\begin{lem}
\label{bd41}
Let $S_n$ denote either of the symbols $d_n$ or $\delta_n$. Then for $1 \leq p, q \leq \infty$ and
$1 \leq n \leq \dim  {\bf E}_{N^{2}}$, we have
\[ S_{n}(\mathbb{B}^{p}_{N}, L_{q}) \leq C N^{s(\frac{1}{p}-\frac{1}{q})}S_{n}(b^{m_N}_{p}, \ell^{m_N}_q), \]
where $m_N  \asymp \dim {\bf E}_{N^{2}} \asymp N^{s}$.
\end{lem}

\begin{proof}
 Using Theorem \ref{prodthm}, we may choose $a_1 > 0$ such that, for any $N$, $f, g \in \Pi_{2N} \Rightarrow
fg \in \Pi_{a_1N}$.  
By Theorems \ref{cubformula} and \ref{genp}, there is an $a_2 > 0$ such that
whenever $\Lambda = \{t_1,\ldots,t_m\}$ is a $\rho$-lattice for $\rho = a_2/N$, then there are constants
$\{w_1,\ldots,w_m\}$ such that, for all $f \in \Pi_{a_1N}$,

\[ \int f = \sum_{j=1}^m w_j f(t_j), \]

and moreover, for all $1 \leq p \leq \infty$, 

\begin{equation}
\label{nrmequiv}
\|f\|_p \sim  N^{-s/p}\|U_N(f)\|_{\ell^m_p} \sim \|U_N(f)\|_{\ell^m_{p,w}},
\end{equation}

where $U_{N}: {\bf E}_{a_{1}^{2}N^{2}} \to \mathbb{R}^{m}$ is given by

\[ U_N(f) = (f(t_1),\ldots,f(t_m)), \]

and for $u = (u_{1},\ldots,u_{m}) \in \mathbb{R}^{m}$,

\[  \|u\|_{\ell^m_{p,w}} = 
\begin{cases}
\left( \sum_{j=1}^m |u_j|^p w_j \right)^{1/p} & \mbox{ if } p \leq \infty \\
\max_{1 \leq j \leq m} |u_j| & \mbox{ if } p = \infty.
\end{cases}
\]

Now, let $\eta$ be a $C^{\infty}$ function on $[0,\infty)$ which equals $1$ on $[0,1]$,
and which is supported in $[0,4]$.  Let $K_{t}$ be the kernel of $\eta(t^{2}{\mathcal L})$, and let
$K^{N} = K_{1/N}$.  Since $\eta(\lambda_k/N^{2}) = 1$ whenever $\lambda_k \leq N^{2} = \omega$
if $N = \omega^{1/2}$, we have that for $f \in  {\bf E}_{\omega}({\mathcal L})$, the
reproducing formula

\begin{equation}
\label{repro}
f(x) = [\eta({\mathcal L}/N^{2})f](x) = \int_{\bf M} K^{N}(x,y) f(y) dy 
\end{equation}
where $dy$ is our invariant measure.  Moreover, $\eta(\lambda_k/N^2) = 0$ if 
$\lambda_k > 4N^{2} = 4\omega$, so that $K^{N}(\cdot,y) \in {\bf E}_{2\omega}
$ for any fixed $y$.  Thus, for any $u = (u_{1},\ldots,u_{m}) \in \mathbb{R}^{m}$, we may
define a map $T: \mathbb{R}^{m} \to {\bf E}_{4N^{2}}$ by
\begin{equation}
\label{tudf}
T(u)(\cdot) = \sum_{j=1}^m w_ju_j K^N(\cdot,t_j).
\end{equation}
We claim that for $1 \leq q \leq \infty$, 
\begin{equation}
\label{lqbdd}
\|T(u)(\cdot)\|_q \leq C\|u\|_{\ell^m_{q,w}}.
\end{equation}
Indeed, if $q = 1$, this follows at once from 
Corollary \ref{Lalphok} with $\alpha = 1$.
For $q = \infty$, (\ref{tudf}) follows from the second equivalence of (\ref{nrmequiv})
in the case $p = 1$, since
\begin{eqnarray*}
\|T(u)\|_{\infty} & \leq & \|u\|_{\infty} \max_{x \in {\bf M}} \sum_{j=1}^m w_j |K^N(x,t_j)|\\
& \leq & C\|u\|_{\infty}\max_{x \in {\bf M}}\int_{\bf M} |K^N(x,y)|dy \\
& \leq & C\|u\|_{\infty},
\end{eqnarray*}
again by Corollary \ref{Lalphok} with $\alpha = 1$.  (\ref{lqbdd}) now follows from the Riesz-Thorin interpolation theorem.

By (\ref{repro}) and the fact that $f, g \in {\bf E}_{4N^{2}} \Rightarrow
fg \in {\bf E}_{a_1^{2}N^{2}}$, we have that for all $f \in {\bf E}_{N^{2}}$,

\[ f(x) = \sum_{j = 1}^m w_j f(t_j) K^N(x,t_j). \]

The rest of the proof of the lemma is basically just as in \cite{brdai}: Note $f = TU_N f$ for $f \in \Pi_N$.
Thus, we can factor the identity $I: {\bf E}_{N^{2}} \cap L^p \to {\bf E}_{\alpha_1^{2}N^{2}} \cap L^q$ as follows:

\[ I:\ \  {\bf E}_N^{2} \cap L^p \stackrel{U_N}{\longrightarrow} \ell^m_p  \stackrel{i_1}{\longrightarrow}
\ell^m_q \stackrel{i_2}{\longrightarrow} \ell^m_{q,w} \stackrel{T}{\longrightarrow} 
{\bf E}_{\alpha_1^{2}N^{2}} \cap L^q,  \]
where $\ell^{m}_{q,w}$ denotes the space $\mathbb{R}^{m}$, equipped with the norm $\|\cdot\|_{\ell^{m}_{q,w}}$,
and $i_{1}, i_{2}$ both denote identity maps.  By well-known properties of $n$-widths (see 
\cite{pin}, Chapter II), 

\[ S_{n}(\mathbb{B}^{p}_{N},L^{q}) \leq S_{n}(\mathbb{B}^{p}_{N},L^{q} \cap {\bf E}_{\alpha_1^{2}N^{2}}) \leq
$$
$$
\|U_{N}\|_{({\bf E}_{N^{2}} \cap L^{p},\ell^{m}_{p})}\|i_{2}\|_{(\ell^{m}_{q},\ell^{m}_{q,w})}\|T\|_{(\ell^{m}_{q,w},{\bf E}_{\alpha_1^{2}N^{2}} \cap L^{q})}
S_{n}(b^{m}_{p},\ell^{m}_{q}). \]

By Theorem \ref{cubformula}, $w_{j }\sim \rho^{s} \sim N^{-s}$ for all $j$, with constants independent of $N$,
so $\|i_2\|_{(\ell^m_q,\ell^m_{q,w})} \sim N^{-s/q}$.  By (\ref{nrmequiv}), 
$\|U_N\|_{({\bf E}_N^{2} \cap L^p,\ell^m_p)} \sim N^{s/p}$, and by (\ref{lqbdd}), $\|T\|_{(\ell^m_{q,w},{\bf E}_{\alpha_1^{2}N^{2}}\cap L^q)} \leq C$.
Combining these facts, we find the lemma.
\end{proof}

{\bf Proof of Theorem \ref{impr}}  
We prove the required upper estimates for $1 \leq p \leq 2 \leq q \leq \infty$, for linear widths; the case of
Kolmogorov widths may be treated similarly.
By the duality
$\delta_n(B^r_p,L_q) = \delta_n(B^r_{q'},L_{p'})$, it suffices to prove them for $1 \leq p \leq 2 \leq q \leq p'$,
which we assume from here on.

By Lemma \ref{bd41} with $q = p'$, and by Lemma \ref{phijest} with $L = \mathcal L$ and $p=q$, we have

%

\begin{equation}
\label{pinwrap}
\delta_n(B^r_p,L_q) \leq \delta_n(B^r_p,L_{p'}) \leq C\sum_{k=0}^{\infty} 2^{-kr} \delta_{n_k}({\mathcal B}^p_{2^{k+1}},L_{p'}) 
\leq 
$$
$$
C\sum_{k=0}^{\infty} 2^{-kr} 2^{ks(\frac{2}{p}-1)}\delta_{n_k}(b_p^{m_k},\ell^{m_k}_{p'}), 
\end{equation}
where $\sum n_k \leq n-1$ and $m_k \asymp 2^{n_k}$.  

Assume now $C_1 2^{sv} \leq n \leq C_1^2 2^{sv}$, with $C_1 > 0$ to be specified later.  We fix a real number $\rho \in (0,2(r/s-1/p))$, and set

\begin{equation}
\label{nkdf}
n_k = 
\begin{cases}
m_k & \mbox{if } 0 \leq k \leq v, \\
\lfloor 2^{s((1+\rho)v-k\rho))} \rfloor & \mbox{if } v < k < (1+\rho^{-1})v, \\
0 & \mbox{if } k \geq (1+\rho^{-1})v.
\end{cases}
\end{equation}

One calculates then that $\sum_k n_k \asymp 2^{sv}$.  Thus one can take $C_1$ so large
that $\sum_k n_k \leq C_1 2^{sv} - 1 \leq n-1$.   The proof is now completed by estimating
the $\delta_{n_k}(b_p^{m_k},\ell^{m_k}_{p'})$ in (\ref{pinwrap}): one has

\begin{equation}
\label{casest}
\delta_{n_k}(b_p^{m_k},\ell^{m_k}_{p'})
\begin{cases}
= 0 & \mbox{if } 0 \leq k \leq v, \\
\leq C2^{-\frac{s(1+\rho)v}{2}}2^{sk(\frac{1}{p'}+\frac{\rho}{2})}((\rho + 1)(k+1-v))^{\frac{1}{2}} & \mbox{if } v < k < (1+\rho^{-1})v, \\
\leq 1 & \mbox{if } k \geq (1+\rho^{-1})v.
\end{cases}
\end{equation}
Here the first case follows by noting that, in that case, $\delta_{n_k}(b_p^{m_k},\ell^{m_k}_{p'}) = \delta_{m_k}(b_p^{m_k},\ell^{m_k}_{p'})$;
the third case follows by noting that, in that case, $\delta_{n_k}(b_p^{m_k},\ell^{m_k}_{p'}) = \delta_{0}(b_p^{m_k},\ell^{m_k}_{p'})$; and the second case
follows from Gluskin's estimate (\cite{glus}) 
$$
\delta_{n_k}(b_p^{m_k},\ell^{m_k}_{p'}) \leq Cm_k^{1/p'}n_k^{-1/2}\log^{1/2}(1+m_k/n_k).
$$
Substituting (\ref{casest}) in (\ref{pinwrap}), one calculates
\[ \delta_n(B^r_p,L_q) \leq C2^{v(-r+s(\frac{1}{p}-\frac{1}{2}))} \leq Cn^{-\frac{r}{s}+\frac{1}{p}-\frac{1}{2}}, \]
as desired.  This completes the proof.

\section{Lower bounds on homogeneous manifolds}

In this section we will prove Theorem \ref{basicshp}. First we need the following simple fact, which is another variant  of our Lemma \ref{balls0}. 
\begin{lem}
\label{balls}
For each positive integer $N$ with $2N^{-1/s} < diam {\bf M}$, there exists a collection of disjoint balls
${\mathcal Q}^N = \{B(x_i^N, N^{-1/s})\}$, such that the balls with the same centers and 3 times the radii cover ${\bf M}$,
and such that $P_N := \#{\mathcal Q}^N \asymp N$.  
\end{lem}
{\bf Proof}  We need only let ${\mathcal Q}^N$ be a maximal disjoint collection of balls of radius $N^{-1/s}$.  Then surely
the balls with the same centers and $3$ times the radii cover ${\bf M}$.  Thus by disjointness
\[ \mu({\bf M}) \geq \sum_{i=1}^{P_N}  \mu(B(x_i^N, N^{-1/s})) \gg \sum_{i=1}^{P_N} 1/N = P_N/N, \]
while by the covering property
\[ P_N/(3^sN) \gg \sum_{i=1}^{P_N} \mu(B(x_i^N, 3N^{-1/s})) \geq \mu({\bf M}) \]
so that $P_N \asymp N$ as claimed.\\

We fix collections of balls ${\mathcal Q}^N$ as in Proposition \ref{balls}.

\begin{lem}
\label{disfns}
Say ${\bf M}$ is a homogeneous manifold.  Then there are smooth functions $\varphi_i^N$ ($2N^{-1/s} < diam {\bf M}$,
$1 \leq i \leq P_N$), as follows:\\
(i) supp $\varphi_i^N \subseteq B_i^N := B(x_i^N, N^{-1/s})$;\\
(ii) For $1 \leq q \leq \infty$, $\|\varphi_i^N\|_q \asymp N^{-1/q}$, with constants independent of $i$ or $N$;\\
(iii) For $1 \leq p \leq \infty$, and $r > 0$, 
$$
\|\sum_{i=1}^{P_N}a_i{\mathcal L}^{r/2}\varphi_i^N\|_p \leq 
CN^{\frac{r}{s}-\frac{1}{p}}\|a\|_p,
$$

with $C$ independent of $a = (a_1,... ,a_{P_N}) \in \mathrm{R}^{P_N}$, $p$ or $N$.
\end{lem}
{\bf Proof} We let $h_0(\xi) = f_0(\xi^2)$ be an even element of ${\mathcal S}(\mathrm{R})$ with supp$\hat{h}_0 \subseteq (-1,1)$.
For a postitive integer $M$ yet to be chosen, let $f(u) =  u^M f_0(u)$, and set $h(\xi) = f(\xi^2) = \xi^{2M}f_0(\xi^2)$,
so that $\hat{h} = c\partial^{2M}\hat{h_0}$ still has support contained in $(-1,1)$.  Thus, by Theorem \ref{wavprop},
there is a $C_0 > 0$ such that
for $t > 0$, the kernel $K_t(x,y)$ of $h(t\sqrt{\mathcal L}) = f(t^2{\mathcal L})$ has the property that
$K_t(x,y) = 0$ whenever $d(x,y) > C_0t$.   Thus if $t = N^{-1/s}/2C_0$, 
\[ \varphi_i^N(x) := \frac{1}{N}K_t(x_i^N,x) \]
satisfies (i).  By Theorem \ref{Lalphimp}, $\|\varphi_i^N\|_q \asymp N^{-1}(N^{-1/s})^{-s/q'} = N^{-1/q}$, so
(ii) holds.  We shall show that (iii) holds if $M$ is sufficiently large.  For this, we will need a technical fact.\\

To state this technical fact, we temporarily suspend the above notation.  
For each positive integer $J$,  we let 
$$
{\mathcal S}_J(\mathrm{R}^+) = \{f \in C^J([0,\infty)): \|f\|_{{\mathcal S}_J}
:= \sum_{i+j \leq J}\|x^i \partial^j f\|_{\infty} < \infty\}.
$$
  Fix $t > 0$.  For $J=J_0$ sufficiently large, using
(\ref{Weyl}), (\ref{lamest}) and (\ref{ulest}), one checks that the right side of
\begin{equation}
\label{expoutm}
K_t^f(x,y) := \sum_l f(t^2\lambda_l)u_l(x)u_l(y)
\end{equation}
converges uniformly to a continuous function on ${\bf M} \times {\bf M}$, and in fact that for
some $C_t > 0$,
\begin{equation}
\label{ktt}
\|K_t^f\|_{\infty} \leq C_t\|f\|_{{\mathcal S}_{J_0}}.
\end{equation}
  By testing
on the $u_m$ as usual, one sees that $K_t^f$ is the kernel of $f(t^2{\mathcal L})$, in the
sense that (\ref{ft2F}) holds for all $F \in L^2$ if $K_t = K_t^f$.   The technical fact that
we need is then that:\\
\ \\
(*) For $J_1$ sufficiently large, Corollaries \ref{kersize} and \ref{Lalphok} continue to hold for 
all $f \in {\mathcal S}_{J_1}$.\\
\ \\
Say that (*) is known, let us revert to the notation of the first paragraph of the proof,
and let us show that (iii) follows for $J$ sufficiently large.  By the Riesz-Thorin interpolation
theorem, we need only do so for $p = 1$ and $\infty$.  If $t = N^{-1/s}/2C_0$, we have
\begin{equation}
\label{kgt}
{\mathcal L}^{r/2}\varphi_i^N = N^{-1}t^{-r}\sum_l (t^2\lambda_l)^{r/2} f(t^2\lambda_l) u_l(x_i^N)u_l(x)
= CN^{\frac{r}{s}-1}K^g_t(x_i^N,x),
\end{equation}
where $g(u) = u^{r/2}f(u)$ and $C$ is independent of $N, i$ or $t$.  Now $g$ may not be in ${\mathcal S}(\mathrm{R}^+)$,
since it might not be smooth at the origin, but if $M$ is sufficiently large, it will be in ${\mathcal S}_{J_0}$, and
thus we may apply Corollaries \ref{kersize} and \ref{Lalphok} to it.  Thus, by (\ref{kgt}) and Corollary \ref{Lalphok},
for $p=1$ we have $\|{\mathcal L}^{r/2}\varphi_i^N\|_1 \leq CN^{\frac{r}{s}-1}$, with $C$ independent of $i,N$.  
(iii) for $p=1$ is an immediate consequence.  As for $p = \infty$, we again set $t = N^{-1/s}/2C_0$.
By Corollary (\ref{kersize}), we have that for any $x$, 
\begin{eqnarray*}
|\sum_{i=1}^{P_N}a_i {\mathcal L}^{r/2}\varphi_i^N(x)| & \leq & 
CN^{\frac{r}{s}-1} \|a\|_{\infty}\sum_{i=1}^{P_N}\frac{t^{-s}}{(1+d(x_i^N,x)/t)^{s+1}}\\
& \leq & CN^{\frac{r}{s}+1} \|a\|_{\infty}\sum_{i=1}^{P_N}\frac{\mu(B_i^N)}{(1+d(x_i^N,x)/t)^{s+1}}\\
& \leq & CN^{\frac{r}{s}+1} \|a\|_{\infty}\int_{\bf M}\frac{dy}{(1+d(y,x)/t)^{s+1}}\\
& \leq & CN^{\frac{r}{s}} \|a\|_{\infty},
\end{eqnarray*}
with $C$ independent of $a, N$, proving (iii).  (In the fourth line we have used the fact that, for 
all $x \in {\bf M}$, all $t > 0$, all $i$ and $N$, and all $y \in B_i^N$,
by the triangle inequality, $(1+d(y,x)/t) \leq C(1+d(x_i^N,x)/t)$ with $C$ independent of $x,y,t,i,N$.  
In the last line we have used (\ref{intest}).)\\

Thus we need only establish the technical fact (*).  In the arguments just given, 
$t = N^{-1/s}/2C_0$ will be less than $1$ except for only finitely many values of $N$.  (iii)
is trivial for those finitely many $N$, so for the purposes of our arguments, we may assume 
$0 < t < 1$.  Thus, for our purposes, we may work in situation (ii) of Corollary \ref{kersize}.
(Situation (i) can be treated similarly.)  We need only establish Corollary
\ref{kersize} for $f \in {\mathcal S}_{J_1}$ for suitable $J_1$ under hypothesis (i), since as we know,
Corollary \ref{Lalphok} is an immediate consequence.

To do this, we let $Z = (0,1) \times {\bf M} \times {\bf M}$, and we let $V$ denote the Banach
space of continuous functions $H$ on $Z$ for which 
$$
\|H\|_V := \sup_{(t,x,y) \in Z} t^s[1+d(x,y)/t]^{s+1}|H(x,y)| 
< \infty.
$$
  For $f \in {\mathcal S}_{J_0}(\bf{R}^+)$, let $H^f(t,x,y) = K_t^f(x,y)$.
By Corollary \ref{kersize}, the linear map $f \to H^f$ takes ${\mathcal S}(\bf{R}^+)$ to $V$; we shall
use the closed graph theorem for Fr\'echet spaces to show that this map is continuous.  Indeed, suppose 
that $f^k \to f$ in ${\mathcal S}(\bf{R}^+)$ and that $H^{f^k} \to H^g$ in $V$, for some
$g \in {\mathcal S}(\bf{R}^+)$.  Then surely $f^k \to f$ in ${\mathcal S}_{J_0}(\bf{R}^+)$,
so by (\ref{ktt}), $H^{f^k} \to H^f$ pointwise; accordingly, $g = f$.  Thus the map is continuous, and
so there is a $C, J_2 \geq J_0$ for which
\begin{equation}
\label{clgrway}
|K_t^f(x,y)| \leq C\|f\|_{{\mathcal S}_{J_2}}\frac{t^{-s}}{(1+d(x,y)/t)^{s+1}}
\end{equation}
for all $t,x,y$ and all $f \in {\mathcal S}(\bf{R}^+)$.

Finally, let $J_1 = J_2 + 1$, and suppose $f \in {\mathcal S}_{J_1}(\bf{R}^+)$.  There is a sequence $f^k$ of
elements of ${\mathcal S}(\bf{R}^+)$ which approaches $f$ in ${\mathcal S}_{J_2}(\bf{R}^+)$.  (Use cutoff functions
and approximate identities.)  Fix $t,x,y$, write (\ref{clgrway}) for $f^k$ in place of $f$, and let $k \to \infty$.
The left sides approach $|K_t^f(x,y)|$, by (\ref{ktt}), while the right sides approaches the right side 
of (\ref{clgrway}).  This proves (*).\\
\ \\
In proving Theorem \ref{basicshp}, we will also obtain lower bounds for the  Gelfand widths $d^n(B^r_p,L_q)$. 

\begin{lem}
\label{lowbd1}
Say $1 \leq p,q \leq \infty$.  If $s_n = d_n$ or $d^n$, then
\begin{equation}
\label{lowbd1way}
s_n(B^r_p,L_q) \geq CN^{-\frac{r}{s}+\frac{1}{p}-\frac{1}{q}}s_n(b_p^{P_N},\ell_q^{P_N})
\end{equation}
for any sufficiently large $n,N$, with $C$ independent of $n,N$.
\end{lem}
\begin{proof}
 With the $\varphi_i^N$ as in Lemma \ref{disfns}, let $H_N$ denote
the space of functions of the form
\begin{equation}
\label{hndf}
g_a = \sum_{i=1}^{P_N} a_i \varphi_i^N,
\end{equation}
for $a = (a_1,\ldots,a_{P_N}) \in \bf{R}^{P_N}$.  By Lemma \ref{disfns} (i) and (ii), 
and the disjointness of the $B_i^N$,
\begin{equation}
\label{gaq}
\|g_a\| \asymp N^{-1/q}\|a\|_q,
\end{equation}
with constants independent of $N$ or $a$.  
By Lemma \ref{disfns} (iii), for some $c > 0$, if we set $\epsilon = \epsilon_N = cN^{-\frac{r}{s}+\frac{1}{p}}$,
and if $a \in \epsilon b_p^{P_N}$, then $g_a \in B^r_p$.  Thus,
\begin{equation}
\label{gndf}
G_N := \{g_a \in H_N: a \in \epsilon b_p^{P_N}\} \subseteq B^r_p.
\end{equation}
For the Gelfand widths, it is a consequence of the Hahn-Banach theorem,
that if $K \subseteq X \subseteq Y$, where $X$ is a subspace of the normed space $Y$, then
$d^n(K,X) = d^n(K,Y)$ for all $n$.  Thus,
\[ d^n(B^r_p,L_q) \geq d^n(G_N,L_q) = d^n(G_N,H_N) \geq CN^{-1/q}d^n(\epsilon_N b_p^{P_N},\ell_q^{P_N})\]
for some $C$ independent of $n,N$, by (\ref{gaq}) and (\ref{gndf}).  This proves the lemma for the 
Gelfand widths.

For the Kolmogorov widths, for the same reason, we need only show that
\begin{equation}
\label{kolgd}
d_n(B^r_p,L_q) \geq Cd_n(G_N,H_N).
\end{equation}
with $C$ independent of $n,N$.

To this end we define the projection operator $Q_N: L_q \to H_N$ by
\[ Q_Nh = g_a,\ \ \ \ \ \ \ \ \ \ \ \ \ \mbox{ where } a_i = \frac{\int h\varphi_i^N}{\|\varphi\|^2_2}.  \]
By Lemma \ref{disfns} (i), (ii) and H\"older's inequality, we have here that each $|a_i| \leq C\|h\chi_i^N\|_q N^{1-1/q'}$,
where $\chi_i^N$ is the characteristic function of $B_i^N$.  By (\ref{gaq}) and the disjointness of the $B_i^N$, we have that
\begin{equation}
\label{qngd}
\|Q_N h\|_q = \|g_a\|_q \asymp N^{-1/q}\|a\|_q \leq cN^{1-1/q-1/q'}\|h\|_q = c\|h\|_q,
\end{equation}
with $C$ independent of $n,N$.  

Accordingly, for any $g \in H_N$ and $h \in L_q$, we have that 
$$
\|g-Q_Nh\|_q = \|Q_Ng-Q_Nh\|_q \leq c\|g-h\|q.
$$
  Thus, if
$K$ is any subset of $H_N$, $d_n(K,L_q) \geq c^{-1}d_n(K,H_N)$.  In particular
\[ d_n(B^r_p,L_q) \geq d_n(G_N,L_q) \geq c^{-1}d_n(G_N,H_N). \]
This establishes (\ref{kolgd}), and completes the proof.\\

{\bf Proof of Theorem \ref{basicshp}}    We will need several facts about widths. First, say $p \geq p_1$, $q \leq q_1$, and $S^n = d_n, d^n$ or 
$\delta_n$.  One then has the following two evident facts
\begin{equation}
\label{pupqdn}
S^n(B^r_p,L_q) \leq CS^n(B^r_{p_1},L_{q_1})
\end{equation}
with $C$ independent of $n$, while
\begin{equation}
\label{pdnqup}
S^n(b_p^M,\ell_q^M) \geq CS^n(b_{p_1}^M,\ell_{q_1}^M)
\end{equation}
with $C$ independent of $n,M$.  

By Lemma \ref{balls}, we may choose $\nu > 0$ such that $P_{\nu n} \geq 2n$ for all sufficiently
large $n$.  In this proof we will always take $N = \nu n$.  We consider the various ranges of $p,q$
separately:

\begin{enumerate}
\item
$q \leq p$.\\
\ \\
In this case, we note that if $S^n = d_n, d^n$ or $\delta_n$, then by (\ref{pupqdn}),
\begin{equation}
\label{pqinf1}
S_n(B^r_p,L_q) \geq CS_n(B^r_{\infty},L_{1}).
\end{equation}
On the other hand, if $s_n = d_n$ or $d^n$, then by (3.1) on page 410 of \cite{LGM},
$s_n(b_{\infty}^{P_N},\ell_{1}^{P_N}) = P_N - n \geq n$.  By this, (\ref{pqinf1}) and
Lemma \ref{lowbd1}, we find that
\[ s_n(B^r_p,L_q) \gg n^{-\frac{r}{s}-1}n = n^{-\frac{r}{s}} \]
first for $s_n = d_n$ or $d^n$ and then for $\delta_n$, by (\ref{linmax}).  This
completes the proof in this case.
\item
$1 \leq p \leq q \leq 2$.\\
\ \\ 
In this case, for the Gelfand widths we just observe, by (\ref{pupqdn}), that
\begin{equation}
\label{pqp1}
d^n(B^r_p,L_q) \geq Cd^n(B^r_{p},L_{p}) \gg n^{-\frac{r}{s}}
\end{equation}
by case 1.  For the Kolmogorov widths we observe, by Lemma \ref{lowbd1} and (\ref{pdnqup}), that
\begin{equation}
\label{pq12}
d_n(B^r_p,L_q) \gg n^{-\frac{r}{s}+\frac{1}{p}-\frac{1}{q}}d_n(b_{p}^{P_N},\ell_q^{P_N}) \gg
$$
$$
n^{-\frac{r}{s}+\frac{1}{p}-\frac{1}{q}}d_n(b_{1}^{P_N},\ell_2^{P_N}) \gg
n^{-\frac{r}{s}+\frac{1}{p}-\frac{1}{q}},
\end{equation}
since, by (3.3) of page 411 of \cite{LGM}, $d_n(b_{1}^{P_N},\ell_2^{P_N}) = \sqrt{1-n/P_N} \geq 1/\sqrt{2}$.
Finally, for the linear widths, we have by (\ref{linmax}), that
\[  \delta_n(B^r_p,L_q) \gg n^{-\frac{r}{s}+\frac{1}{p}-\frac{1}{q}}. \]
This completes the proof in this case.
\item
$2 \leq p \leq q$.\\
\ \\ 
In this case, for the Kolmogorov widths we just observe, by (\ref{pupqdn}), that
\begin{equation}
\label{pqp1}
d_n(B^r_p,L_q) \geq Cd_n(B^r_{p},L_{p}) \gg n^{-\frac{r}{s}}
\end{equation}
by case 1.  For the Gelfand widths we observe, by Lemma \ref{lowbd1} and (\ref{pdnqup}), that
\begin{equation}
\label{pq12}
d^n(B^r_p,L_q) \gg n^{-\frac{r}{s}+\frac{1}{p}-\frac{1}{q}}d^n(b_{p}^{P_N},\ell_q^{P_N}) \gg
$$
$$
n^{-\frac{r}{s}+\frac{1}{p}-\frac{1}{q}}d^n(b_{2}^{P_N},\ell_{\infty}^{P_N}) \gg
n^{-\frac{r}{s}+\frac{1}{p}-\frac{1}{q}},
\end{equation}
since, by (3.5) of page 412 of \cite{LGM}, 
$$
d^n(b_{2}^{P_N},\ell_{\infty}^{P_N}) = \sqrt{1-n/P_N} \geq 1/\sqrt{2}.
$$
Finally, for the linear widths, we have by (\ref{linmax}), that
\[  \delta_n(B^r_p,L_q) \gg n^{-\frac{r}{s}+\frac{1}{p}-\frac{1}{q}}. \]
This completes the proof in this case.
\item
$1 \leq p \leq 2 \leq q \leq \infty$.\\
\ \\
Say $1 \leq \alpha \leq \alpha_1 \leq \infty$.  By H\"older's inequality,
\begin{equation}
\label{hold1}
\|a\|_{\alpha} \leq M^{\frac{1}{\alpha}-\frac{1}{\alpha_1}}\|a\|_{\alpha_1} 
\end{equation}
if $a \in \mathrm{R}^M$.  This implies that
\begin{equation}
\label{hold2}
b_{\alpha_1}^M \subseteq M^{\frac{1}{\alpha_1}-\frac{1}{\alpha}}b_{\alpha}^M.
\end{equation}
From Lemma \ref{lowbd1}, (\ref{pdnqup}) and (\ref{hold1}), we find that
\begin{equation}
\label{1p2qdn1}
d_n(B^r_p,L_q) \gg n^{-\frac{r}{s}+\frac{1}{p}-\frac{1}{q}}d_n(b_p^{P_N},\ell_q^{P_N})
\gg 
$$
$$
n^{-\frac{r}{s}+\frac{1}{p}-\frac{1}{q}}d_n(b_1^{P_N},\ell_q^{P_N})
\gg n^{-\frac{r}{s}+\frac{1}{p}-\frac{1}{2}}d_n(b_1^{P_N},\ell_2^{P_N})
\gg n^{-\frac{r}{s}+\frac{1}{p}-\frac{1}{2}}.
\end{equation}
From Lemma \ref{lowbd1}, (\ref{pdnqup}) and (\ref{hold2}), we find that
\begin{equation}
\label{1p2qdn2}
d^n(B^r_p,L_q) \gg n^{-\frac{r}{s}+\frac{1}{p}-\frac{1}{q}}d^n(b_p^{P_N},\ell_q^{P_N})
\gg 
$$
$$
n^{-\frac{r}{s}+\frac{1}{p}-\frac{1}{q}}d^n(b_p^{P_N},\ell_{\infty}^{P_N})
\gg n^{-\frac{r}{s}+\frac{1}{2}-\frac{1}{q}}d^n(b_2^{P_N},\ell_{\infty}^{P_N})
\gg n^{-\frac{r}{s}+\frac{1}{2}-\frac{1}{q}}.
\end{equation}
Finally, from (\ref{1p2qdn1}), (\ref{1p2qdn2}) and (\ref{linmax}),
\begin{equation}
\label{1p2qdn3}
\delta_n(B^r,L_q) \gg \max(n^{-\frac{r}{s}+\frac{1}{p}-\frac{1}{2}},n^{-\frac{r}{s}+\frac{1}{2}-\frac{1}{q}}).
\end{equation}
This completes the proof.

\end{enumerate}
\end{proof}

\end{document}